\documentclass[a4paper,11pt]{article}
\usepackage{stmaryrd}
\usepackage{amsmath,amssymb,amsfonts}
\usepackage{times}
\usepackage[T1]{fontenc}
\usepackage{graphicx}
\usepackage[english]{babel}
\usepackage{amsmath}
\usepackage{amsthm}
\usepackage{amssymb}
\usepackage{enumerate}
\usepackage{xspace}
\usepackage{euscript}
\usepackage{graphicx}
\usepackage{amscd}
\usepackage{epsfig}
\usepackage{tabularx}
\usepackage{color}
        \newtheorem{lemma}{Lemma}[section]
        
        \newtheorem{theorem}[lemma]{Theorem}
         \newtheorem{corollary}[lemma]{Corollary}
        \newtheorem{definition}[lemma]{Definition}

\title{Finite Element Approximation of Transmission Eigenvalues for Anisotropic Media
\thanks{JS's work was partially supported by the National Science Foundation under Grant No. DMS-1521555. CZ's work was supported by Natural Science Foundation of Xinjiang Autonomous Region under No. 50500-600112096, and National Natural Science Foundation under No. 11771248. JS is partially supported by NSF DMS-1521555}}
\author{Bo Gong \thanks{Beijing Computational Science Research Center, Beijing 100193, China. ({\tt gongbo@csrc.ac.cn})} \and
Jiguang Sun \thanks{Department of
Mathematical Sciences, Michigan Technological University, Houghton, MI 49931, U.S.A. ({\tt  jiguangs@mtu.edu}).}
\and Tiara Turner 
\thanks{Department of Mathematics and Computer Science, University of Maryland Eastern Shore, Princess Anne, MD 21853, U.S.A. ({\tt tdturner@umes.edu}).}
\and Chunxiong Zheng \thanks{College of Mathematics and Systems Science, Xinjiang University, Urumqi 830046, China, and Department of Mathematical Sciences, Tsinghua University, Beijing, 100084, P.R.China ({\tt czheng@tsinghua.edu.cn}).}
}
\date{}
\begin{document}
\maketitle

\begin{abstract}
The transmission eigenvalue problem arises from the inverse scattering theory for inhomogeneous media and
has important applications in many qualitative methods. The problem is posted as a system of two second
order partial differential equations and is essentially nonlinear, non-selfadjoint, and of higher order.
It is nontrivial to develop effective numerical methods and the proof of convergence is challenging. In this paper, we formulate 
the transmission eigenvalue problem for anisotropic media as an eigenvalue problem of a holomorphic
Fredholm operator function of index zero. The Lagrange finite elements are used for discretization and the convergence is proved
using the abstract approximation theory for holomorphic operator functions. A spectral indicator method is
developed to compute the eigenvalues. Numerical examples are presented for validation.
\end{abstract}

\section{Introduction}
The transmission eigenvalue problem arises from the inverse scattering theory for inhomogeneous media and
has important applications in many qualitative methods \cite{CakoniETAL2010IP, CakoniHaddar2013}. 
It was shown that the transmission eigenvalues can be reconstructed
from the scattering data and used to obtain physical properties of the unknown target. There is a practical need 
to compute the transmission eigenvalues effectively and efficiently. Furthermore, the problem
is nonlinear and non-selfadjoint. It is worthwhile to study such problems from the numerical analysis point of view.

Numerical approximations for transmission eigenvalues have 
been an active research topic since the first paper by Colton, Monk and Sun \cite{ColtonMonkSun2010}. 
Many methods have been proposed including the
conforming finite element methods \cite{ColtonMonkSun2010, Sun2011SIAMNA,  CakoniMonkSun2014CMAM},
the mixed finite element methods \cite{ColtonMonkSun2010, Ji2012,  Yang2016, ChenH2017, Camano2018},
the non-conforming finite element methods \cite{YangHanBi2016CMAME},
the discontinuous Galerkin methods \cite{GengJiSunXu2016JSC, YangBiLiHan2017JCAM}, 
the virtual element method \cite{MoraVelasquez2018M3AS}, 
the spectral element methods \cite{AnShen2013JSC, An2016ANM}, 
the collocation method using the fundamental solutions \cite{Kleefeld2018IP, Kleefeld2018IPB}
and the boundary integral equation methods \cite{CossonniHaddar2013, Kleefeld2013IP, ZengSunXu2016SCM,CakoniKress2017}.
In addition, multilevel/multigrid methods and numerical linear algebra techniques have also been proposed 
\cite{JiSunXie2014JSC,LiEtal2014JSC,LiEtal2017IP, XieWu2017JSC}.

In this paper, we consider the finite element approximation of the transmission eigenvalue problem for anisotropic media.
Let $D \subset \mathbb R^2$ be a bounded Lipschitz domain.
Let $A(x)$ be a $2\times 2$ matrix valued function with $L^{\infty}(D)$ entries and $n(x) \in L^{\infty}(D)$. 
Assume that $n(x) > 0$ is bounded and  $A(x)$ is symmetric such that $  \xi \cdot \mbox{Im}(A) \xi \le 0$
and $  \xi \cdot \mbox{Re}(A) \xi \ge \gamma |\xi|^2$ for all $\xi \in \mathbb R^2$ with
$\gamma > 0$.
The transmission eigenvalue problem is to find $\eta \in \mathbb C$ and non-trivial functions $u, v$ such that
\begin{subequations}\label{ATE}
\begin{align}
\label{ATEw}&\nabla \cdot A \nabla u+\eta nu=0, &\text{in } D,\\[1mm]
\label{ATEv}&\Delta v+ \eta v=0, &\text{in } D,\\[1mm]
\label{ATEbcD}&u -v = 0, &\text{on } \partial D, \\[1mm]
\label{ATEbcN}& \partial_A u - \partial_\nu v= 0, &\text{on } \partial D,
\end{align}
\end{subequations}
where $\nu$ is the unit outward normal to $\partial D$ and $\partial_A u$ is the conormal derivative
\[
\partial_A u(x):= \nu({x}) \cdot A({x}) \nabla u({x}), \quad {x} \in \partial D.
\]


It is nontrivial to prove the convergence of the finite element methods for \eqref{ATE}
due to the nonlinearity. The classical spectral convergence theory for linear compact operators cannot 
be applied directly \cite{BabuskaOsborn1991, Boffi2010AN, SunZhou2016}. 
The existing convergence results only cover the isotropic media, i.e., $A=I$. In this case, 
the transmission eigenvalue problem can be reformulated as a non-linear 
fourth-order eigenvalue problem. Note that $\partial_A u = \partial_\nu u$ if $A=I$. 
Introducing $w=u-v$ and subtracting \eqref{ATEv} from \eqref{ATEw}, \eqref{ATE} can be written as a nonlinear fourth order eigenvalue problem
of finding $\eta$ and $w$ such that $w=\partial_\nu w = 0$ on $\partial D$ and 
\begin{equation}\label{TE4orderx}
\left[ (\triangle +\eta n(x))\frac{1}{n(x)-1}(\triangle + \eta)\right]w = 0 \quad \text{in } D.
\end{equation}
In \cite{Sun2011SIAMNA}, \eqref{TE4orderx} is recasted as the combination of a linear fourth order eigenvalue problem, 
which can be solved using a conforming finite element, and a nonlinear algebraic equation whose roots are transmission eigenvalues. 
In \cite{CakoniMonkSun2014CMAM}, the authors introduce a new variable and obtain a mixed formulation for \eqref{TE4orderx} 
consisting of one fourth order equation and one second order equation.
Then the convergence of a mixed finite element method is obtained using the perturbation theory for eigenvalues of nonselfadjoint
compact operators.

However, the above technique does not work for the anisotropic media since \eqref{TE4orderx} is not available if $A \ne I$. 
There exist a few numerical methods to compute transmission eigenvalues of anisotropic media \cite{JiSun2013JCP}.
Unfortunately, none of them provide a rigorous convergence proof.
In this paper, we reformulate \eqref{ATE} as an eigenvalue
problem of a holomorphic operator function. Then Lagrange finite elements and the spectral projection are used 
to compute the eigenvalues inside a region on the complex plane.  
Using the classic finite element theory  \cite{BrennerScott2008} and the
approximation results for the eigenvalues of holomorphic Fredholm operator functions \cite{Karma1996a, Karma1996b, Beyn2014}, 
we prove that the convergence of the finite element approximation.

The proposed method has several characteristics: 
1) the transmission eigenvalue problem of anisotropic media is reformulated as the eigenvalue problem of a holomorphic Fredholm operator function;
2) simple Lagrange finite elements can be used for discretization;
3) a rigorous convergence proof for transmission eigenvalue problem of anisotropic media is obtained for the first time to the authors' knowledge;
and 4) the method can be easily extended to the Maxwell's transmission eigenvalue problem and the elastic transmission
eigenvalue problem. 

The rest of the paper is arranged as follows. In Section 2, preliminaries of holomorphic Fredholm operator functions
and the abstract approximation theory for the eigenvalue problem are presented. In Section 3, we reformulate the
transmission eigenvalue problem as the eigenvalue problem of an operator function, which is holomorphic. Section 4 contains the Lagrange finite element 
discretization of the operator eigenvalue problem and its convergence proof. In Section 5, a spectral indicator method is 
designed to compute the eigenvalues in a region on the complex plane. Numerical results are presented in Section 6. 
Finally, in Section 7, we make some conclusions and discuss future work.

\section{Preliminaries}
We present the preliminaries of the approximation theory for eigenvalues of holomorphic Fredholm operator functions
\cite{Karma1996a, Karma1996b, Beyn2014}.
Let $X$ and $Y$ be complex Banach spaces. Denote by $\mathcal{L}(X, Y)$ the space of bounded linear operators from $X$ to $Y$.
Let  $\Omega \subseteq \mathbb C$ be a compact simply connected region. 
Let $T: \Omega \to \mathcal{L}(X, Y)$ be a holomorphic operator function on $\Omega$.
\begin{definition}
A bounded linear operator $T: X \to Y$ is said to be Fredholm if
\begin{itemize}
\item[1.] the subspace $\mathcal{R}(T)$, range of $T$, is closed in $Y$;
\item[2.] the subspace $\mathcal{N}(T)$, null space of $T$, and $Y/\mathcal{R}(T)$ are finite-dimensional.
\end{itemize}
The index of $T$ is the integer defined by
\[
\text{ind}(T) = \dim \mathcal{N}(T) - \dim (Y/\mathcal{R}(T)).
\]
\end{definition}

In the rest of the paper, we assume that $T(\eta)$ is a holomorphic operator function and, for each $\eta \in \Omega$, 
$T(\eta)$ is a Fredholm operator of index $0$. 
\begin{definition}A complex number $\lambda \in \Omega$ is called an eigenvalue
of $T$ if there exists a nontrivial $x \in X$ such that $T(\lambda) x = 0$. The element $x$ is called an eigenelement associated with $\lambda$.
\end{definition}

The resolvent set $\rho(T)$ and the spectrum $\sigma(T) $ of $T$ are defined as
\begin{equation}\label{rhoT}
\rho(T) = \{\eta \in \Omega: T(\eta)^{-1} \text{ exists and is bounded}\}
\end{equation}
and
\begin{equation}\label{sigmaT}
\sigma(T) = \Omega \setminus \rho(T),
\end{equation}
respectively. Since $T(\eta)$ is holomorphic, the spectrum $\sigma(T)$ has no cluster points in $\Omega$ and every $\lambda \in \sigma(T)$ is
an eigenvalue for $T(\eta)$. Furthermore,  the operator valued function
$T^{-1}(\cdot)$ is meromorphic (see Section 2.3 of \cite{Karma1996a}).
The dimension of $\mathcal{N}(T(\lambda))$ of an eigenvalue $\lambda$ is called the 
geometric multiplicity. 

\begin{definition}
An ordered sequence of elements $x_0, x_1, \ldots, x_k$ in $X$ is called a Jordan
chain of $T$ at an eigenvalue $\lambda$ if
\[
T(\lambda)x_j + \frac{1}{1!} T^{(1)}(\lambda)x_{j-1}+ \ldots + \frac{1}{j!} T^{(j)}(\lambda) x_0 = 0, \quad j = 0, 1, \ldots, k,
\]
where $T^{(j)}$ denotes the $j$th derivative. 
\end{definition}

The length of any Jordan chain of an eigenvalue is finite.
Denote by $m(T,\lambda, x_0)$ the length of a Jordan chain formed by an eigenelement $x_0$.
The maximal length of all Jordan chains of the eigenvalue $\lambda$ is denoted by $\kappa(T, \lambda)$.
Elements of any Jordan chain of an eigenvalue $\lambda$ are called generalized eigenelements of $\lambda$.
\begin{definition}
The closed linear hull of all generalized eigenelements of an eigenvalue $\lambda$, denoted by $G(\lambda)$,
is called the generalized eigenspace of $\lambda$.
\end{definition}

A basis $x_0^1, \ldots, x_0^J$ of the eigenspace of eigenvalue $\lambda$, i.e., $\mathcal{N}(T(\lambda))$, is called canonical if
\begin{itemize}
\item[(i)] $m(T, \lambda, x_0^1) = \kappa(T, \lambda)$
\item[(ii)] $x_0^j$ is an eigenelement of the maximal possible order belonging to some direct complement $M_j$
		in $\mathcal{N}(T(\lambda))$ to the linear hull $\text{span}\{x_0^1, \ldots, x_0^{j-1} \}$, i.e.,
		\[
			m(T, \lambda, x_0^j) = \max_{x \in M_j} m(T, \lambda, x) \quad \text{for } j = 2, \ldots, J.
		\]
\end{itemize}
The numbers
\[
m_i(T, \lambda) : = m(T, \lambda, x_0^i) \quad \text{for } j=2, \ldots, J
\]
are called the partial multiplicities of $\lambda$. The number
\[
m(\lambda) := \sum_{i=1}^J m_i(T, \lambda)
\]
is called the algebraic multiplicity of $\lambda$ and coincides with the dimension of the generalized eigenspace $G(\lambda)$.


Let $X_n, Y_n$ be Banach spaces, not necessarily subspaces of $X, Y$.
Denote by $\Phi_0(X, Y)$ the sets in $\mathcal{L}(X, Y)$ of all Fredholm operators and with index zero.
Consider a sequence of holomorphic Fredholm operator functions 
\[
T_n: \Omega \to  \Phi_0(X_n, Y_n), \quad n \in \mathbb N.
\]
Assume that the following approximation properties hold.
\begin{itemize}
\item[A1.] There exist linear bounded mapping $p_n \in \mathcal{L}(X, X_n)$,
		$q_n \in \mathcal{L}(Y, Y_n)$ such that
		\begin{eqnarray*}
			&&\lim_{n \to \infty} \|p_n x\|_{X_n}=\|x\|_X, \quad x \in X, \\
			&&\lim_{n \to \infty} \|q_n y\|_{Y_n}=\|y\|_Y, \quad y \in Y.
		\end{eqnarray*}
\item[A2.] The sequence $\{T_n\}$ satisfies
		 \[
		 	\sup_{n \in N} \sup_{\eta \in \Omega} \|T_n(\eta)\| < \infty.
		 \]
\item[A3.] $T_n(\eta)$ converges regularly to $T(\eta)$ for all $\eta \in \Omega$, i.e., 
		\begin{itemize}
		\item[(a)] $\lim_{n \to \infty} \|T_n(\eta) p_n x - q_n T(\eta) x\|_{Y_n} = 0, x \in X$,
		\item[(b)] for any subsequence $x_n \in X_n, n \in N' \subset \mathbb N$ with
			$\|x_n\|_{X_n}, n \in N'$ bounded and 
			$\lim_{N' \ni n \to \infty} \|T_n(\eta)x_n - q_ny\|_{Y_n} = 0$ for
			some $y \in Y$, there exists a subsequence $N'' \subset N'$ such that
			\[
				\lim_{N'' \ni n \to \infty} \|x_n - p_n x\|_{X_n} = 0.
			\]
		\end{itemize}
\end{itemize}

If the above conditions are satisfied, one has the following abstract approximation result
(see, e.g., Theorem 2.10 of \cite{Beyn2014} or Section 2 of \cite{Karma1996b}). 
\begin{theorem}\label{Thm210} 
Assume that (A1)-(A3) hold. 
 For any $\lambda \in \sigma(T)$ there exists $n_0 \in {\mathbb N}$ and a sequence $\lambda_n \in \sigma(T_n)$,
	$n \ge n_0$, such that $\lambda_n \to \lambda$ as $n \to \infty$. For any sequence $\lambda_n \in \sigma(T_n)$
	with this convergence property, 
	one has that
	\[
		|\lambda_n - \lambda| \le C \epsilon_n^{1/\kappa}, 
	\]
	where
	\[
		\epsilon_n = \max_{|\eta-\lambda| \le \delta}\max_{v \in G(\lambda)} \| T_n(\eta)p_n v - q_n T(\eta) v\|_{X_n},
	\]
	for sufficiently small $\delta > 0$.
\end{theorem}

\section{Transmission Eigenvalue Problem}
In this section, we reformulate the transmission eigenvalue problem \eqref{ATE} as an eigenvalue problem of a holomorphic operator function.
To this end, consider the following Helmholtz equation with Robin boundary condition. 
Given a function $g \in L^2(\partial D)$, find $u$ such that
\begin{subequations}\label{HelmholtzA}
\begin{align}
\label{HelmholtzAE}\nabla \cdot A \nabla u+\eta n u&=0,\quad&\mathrm{in}\quad D,\\[1mm]
\label{HelmholtzAB}{\partial}_A u -iu&=g,\quad&\mathrm{on}\quad\partial D.
\end{align}
\end{subequations}
The weak form is to find $u\in H^1(D)$ such that
\begin{equation}\label{HelmholtzAW}
(A\nabla u,\nabla {v})_D-i\langle u, {v}\rangle_{\partial D}-\eta(nu, {v})_D=\langle g, {v}\rangle_{\partial D}\quad \text{for all } v\in H^1(D).
\end{equation}
Let $\mathbb{C}_0^+=\{\eta\in\mathbb{C}:\Im\eta\ge 0\}$, where $\Im\eta$ denotes the imaginary part of $\eta$. We have the
following well-posedness result for \eqref{HelmholtzAW}. Its proof is provided for later reference.
\begin{lemma}\label{LeAnH}
If $\eta \in\mathbb{C}_0^+$, then \eqref{HelmholtzAW} has a unique solution $u\in H^1(D)$ for $g\in L^2(\partial D)$.
Furthermore,
\[
\|u\|_{H^1(D)} \le C \|g\|_{L^2(\partial D)}.
\]
\end{lemma}
\begin{proof}
Define
\[
a(u, v) := (A\nabla u, \nabla {v})_D-i\langle u, {v}\rangle_{\partial D}-\eta(nu, {v})_D.
\]
It is easy to verify that $a(u, v)$ satisfies the G{\aa}rding's inequality \cite{BrennerScott2008}, 
i.e., there exist $K>0$ large enough and $\alpha_0 > 0$ such that
\begin{equation}\label{Garding}
\text{Re}\left\{a(v, v)\right\} + K\|v\|^2_{L^2(D)} \ge \alpha_0 \|v\|^2_{H^1(D)} \quad \text{for all } u \in H^1(D).
\end{equation}

Hence, it suffices to prove the uniqueness. If $u$ is the solution for $g=0$, then by setting $v={u}$ one has that
\[
(A\nabla u,\nabla {u})_D-i\|u\|_{L^2(\partial D)}^2-\eta\|n^{\frac{1}{2}}u\|_{L^2(D)}^2=0.
\]
The imaginary part of the above equation is simply
\[
-\|u\|_{L^2(\partial D)}^2-\Im\eta\|n^{\frac{1}{2}}u\|_{L^2(D)}^2\ge0,
\]
which implies $u=0$ on $\partial D$. Therefore, we have $\partial_{A}u=0$. By
the unique continuation theorem, we have $u=0$ on $D$.

Let the coercive sesquilinear form $a_+(\cdot, \cdot)$ be given by
\begin{equation}\label{aplus}
a_+(u, v) := a(u,v)+K(u,{v}).
\end{equation}
Define a compact operator $\mathcal{K}: L^2(D) \to H^1(D)$ such that $\mathcal{K} u$ solves the following equation
\[
a_+(\mathcal{K} u, v) = -K(u, {v})_D \quad \text{for all } v \in H^1(D).
\]
Similarly, there exists a unique $f \in H^1(D)$ such that 
\[
a_+(f, v) = \langle g, {v}\rangle_{\partial D}\quad \text{for all } v\in H^1(D).
\]
If $u$ solves \eqref{HelmholtzAW}, it satisfies
\[
a_+(u, v) = K(u, {v})_D +  \langle g,{v}\rangle_{\partial D}\quad \text{for all } v\in H^1(D).
\]
Hence $u$ satisfies the operator equation
\begin{equation}\label{IKuf}
(I+\mathcal{K}) u = f.
\end{equation}
The Fredholm alternative (see, e.g., Theorem 1.1.12 of \cite{SunZhou2016}) leads to
\[
\|u\|_{H^1(D)} \le C \|f\|_{H^1(D)} \le C \|g\|_{L^2(\partial D)}
\]
and the proof is complete.
\end{proof}

Consequently, we have a solution operator 
\[
S_1(\eta): L^2(\partial D) \to H^1(D) \quad \text{such that }\, u:=S_1(\eta) g.
\]

\begin{theorem}\label{SH}
The operator $S_1: \mathbb C_0^+ \to \mathcal{L}(L^2(\partial D), H^1(D))$ is holomorphic.
\end{theorem}
\begin{proof}
Let $\eta, \eta + \delta \eta \in \Omega$. For a fixed $g \in L^2(\partial D)$, let $u$ be the solution of
\[
(A\nabla u,\nabla {v})_D-i\langle u,{v}\rangle_{\partial D}-\eta(nu,{v})_D=\langle g,{v}\rangle_{\partial D}\quad \text{for all } v\in H^1(D),
\]
and $w$ be the solution of
\[
(A\nabla w,\nabla {v})_D-i\langle w, {v}\rangle_{\partial D}-(\eta+\delta \eta)(nw,{v})_D=\langle g,{v}\rangle_{\partial D}\quad \text{for all } v\in H^1(D).
\]
Then one has that
\[
(A\nabla (w-u),\nabla {v})_D-i\langle (w-u),{v}\rangle_{\partial D}-\eta(n(w-u),{v})_D= \delta \eta (n w, {v})_D\quad \text{for all } v\in H^1(D).
\]
The above problem has a unique solution $w-u \in H^1(D)$.
Let $\phi$ be the solution of
\begin{equation}\label{PhiA}
(A\nabla \phi,\nabla {v})_D-i\langle \phi, {v}\rangle_{\partial D}-\eta(n\phi,{v})_D= \delta \eta (n u, {v})_D\quad \text{for all } v\in H^1(D).
\end{equation}
Then
\begin{eqnarray*}
 \|w-u -\phi\|_{H^1(D)} &\le& C |\delta \eta| \|u-w\|_{L^2(D)} \\
 	&\le& C |\delta \eta|^2 \|w\|_{L^2(D)}\\
	&\le& C |\delta \eta|^2 \|g\|_{L^2(\partial D)}.
\end{eqnarray*}
Hence $S_1(\eta)g$ is holomorphic on $\mathbb C_0^+$ and thus $S_1(\eta)$ is holomorphic by Theorem 1.7.1 of \cite{Gohberg2009}.
\end{proof}

Define an operator $S_0:L^2(\partial D) \to H^1(D)$ such that $v:=S_0(\eta) g$
solves \eqref{HelmholtzA} with $A=I$ and $n(x)= 1$. In this case, $\partial_I v$ is simply $\partial_\nu v$. 
Due to Lemma \ref{LeAnH}, $S_j(\eta), j=0, 1$, which maps $g\in L^2(\partial D)$ to 
$u|_{\partial D} \in H^{\frac{1}{2}}(\partial D)$ is bounded for any $\eta\in\mathbb{C}$ with $\Im \eta\ge 0$.
Therefore, there exists a neighborhood $\hat{\mathbb{C}}$ of $\mathbb{C}_0^+=\{\eta\in\mathbb{C}:\Im\eta\ge 0\}$,
such that $S_j(\eta), j=1,2,$ is holomorphic in $\hat{\mathbb{C}}$.

Let $\Omega \subset \hat{\mathbb{C}}$ be a compact set. Consider the operator function 
\[
T: \Omega \to  \mathcal{L}(L^2(\partial D),  L^2(\partial D))
\]
 defined by
\begin{equation}\label{Teta}
T(\eta)=\mathcal{I}[S_1(\eta)-S_0(\eta)],
\end{equation}
where $\mathcal{I}$ is the trace operator from $H^1(D)$ into $L^2(\partial D)$.

{\bf Remark:} The operators $\mathcal{I}S_1$ and  $\mathcal{I}S_0$ are the Robin-to-Dirichlet operators.
Under the assumptions that $A=I$ and $n(x)$ is a constant, 
a similar formulation is proposed in \cite{CakoniKress2017} using the boundary integral equation method for
the transmission eigenvalue problem.

\begin{lemma}
The operator function $T(\eta)$ is holomorphic in $\hat{\mathbb{C}}$. 
\end{lemma}
\begin{proof}
It is clear form the proof of  Theorem~\ref{SH} that $\mathcal{I} S_1(\eta)$ is holomorphic.
Consequently, $T(\eta)$ is holomorphic in $\hat{\mathbb{C}}$.
\end{proof}

\begin{theorem}
A complex number $\lambda$ is an eigenvalue of $T$ if and only if it is a transmission eigenvalue of \eqref{ATE}.
\end{theorem}
\begin{proof}
Let $\lambda$ be an eigenvalue of $T$ and $g$ is such that $T(\lambda) g = 0$. 
Then let $u:=S_1(\lambda) g$ be the solution of 
\begin{subequations}\label{HelmholtzAL}
\begin{align}
\label{HelmholtzALE}\nabla \cdot A \nabla u+\lambda n u&=0,\quad&\mathrm{in}\quad D,\\[1mm]
\label{HelmholtzALB}{\partial}_A u -iu&=g,\quad&\mathrm{on}\quad \partial D.
\end{align}
\end{subequations}
and $v:=S_0(\lambda) g$ be the solution of 
\begin{subequations}\label{HelmholtzALI}
\begin{align}
\label{HelmholtzALIE}\nabla \cdot \nabla v+\lambda v&=0,\quad&\mathrm{in}\quad D,\\[1mm]
\label{HelmholtzALIB}{\partial}_\nu v -iv&=g,\quad&\mathrm{on}\quad\partial D.
\end{align}
\end{subequations}
Thus one has that
\[
{\partial}_A u -iu = {\partial}_\nu v -iv \quad \text{on } \partial D.
\]
Moreover, $T(\lambda) g = 0$ implies that
\[
u = v \quad \text{on } \partial D.
\]
Thus $(\lambda, u, v)$ satisfies \eqref{ATE}.

On the other hand side, if $(\lambda, u, v)$ satisfies \eqref{ATE}, one has that
\[
 {\partial}_A u -iu = {\partial}_\nu v -iv
\]
due to \eqref{ATEbcD} and \eqref{ATEbcN}. Let 
\[
g:=  {\partial}_A u -iu.
\]
Then $u=S_1(\lambda)g$ and $v=S_0(\lambda) g$. Using \eqref{ATEbcD}, one has that
\[
T(\lambda) g = \mathcal{I} (S_1(\lambda)g-S_0(\lambda) g) = \mathcal{I} (u -v) = 0.
\]
The proof is complete.
\end{proof}

%
%

{\bf Remark} There are other ways to formulate the transmission eigenvalue problem as an operator eigenvalue problem \eqref{ATE}. 
For example, consider the problem of finding $u$ such that
\begin{subequations}\label{HelmholtzB}
\begin{align}
\nabla \cdot A \nabla u+\eta n u&=0,\quad&\mathrm{in}\quad D,\\[1mm]
u&=g,\quad&\mathrm{on}\quad\partial D.
\end{align}
\end{subequations}

Recall that $\lambda$ is called a modified Dirichlet eigenvalue if
there exists a nontrivial solution $w$ to
\begin{eqnarray*}
\nabla \cdot A \nabla w+\lambda n w&=&0,\quad\mathrm{in}\quad D,\\
w&=&0,\quad\mathrm{on}\quad\partial D.
\end{eqnarray*}
In the case of $A=I$ and $n(x) = 1$, $\lambda$ is simply a Dirichlet eigenvalue.

If $\eta$ is neither a modified Dirichlet eigenvalue nor a Dirichlet eigenvalue, 
there exists a solution $u \in H^1(D)$. One has the Dirichlet-to-Neumann operator 
$\hat{S}_1: L^2(\partial D) \to H^{-1/2}(\partial D)$ such that
\[
\hat{S}_1(\eta) g = \frac{\partial u}{\partial \nu}.
\]
Consequently, the operator eigenvalue problem is to find $\eta$ and $g \ne 0$ such that
\[
\hat{T}(\eta) g := \left(\hat{S}_1(\eta) -\hat{S}_0(\eta) \right)g =0,
\]
where $\hat{S}_0$ is the Dirichlet-to-Neumann operator for \eqref{HelmholtzB} with $A=I$ and $n = 1$. 
Note that the requirement that $\eta$ can not be a modified Dirichlet eigenvalue or a Dirichlet eigenvalue could
generate unnecessary complications in the analysis and computation 
(see \cite{CossonniHaddar2013, ZengSunXu2016SCM, CakoniKress2017}).

%

\section{Finite Element Approximation}
In this section, we propose a finite element approximation $T_h(\eta)$ for $T(\eta)$. 
Let $\mathcal{T}_h$ be a regular triangular mesh for $D$ with mesh size $h$.
For simplicity, let $V_h \subset H^1(D)$ be the linear Lagrange finite element space associated with $\mathcal{T}_h$
and $V^B_h := \{v_h|_{\partial D}, v_h \in V_h\}$ be the restriction of $V_h$ on $\partial D$. It is clear that $V^B_h \subset L^2(\partial D)$.

The finite element formulation for \eqref{HelmholtzAW} is to find $u_h \in V_h$ such that
\begin{equation}\label{HADWeak}
(A\nabla u_h,\nabla v_h)_\Omega-i\langle u_h, v_h\rangle_\Gamma-\eta(nu_h, v_h)_\Omega=\langle p_h g,v_h\rangle_{\partial D} \quad \text{for all } v_h \in V_h,
\end{equation}
where $p_h: L^2(\partial D) \to V_h^B$ is the $L^2$ projection such that
\begin{equation}\label{ph}
\langle g, v_h \rangle_{\partial D} = \langle p_h g, v_h \rangle_{\partial D} \quad \text{for all } v_h \in V_h^B.
\end{equation}

\begin{lemma} Let $\eta \in \mathbb C_0^+$.
There exists a unique solution $u_h$ to \eqref{HADWeak}.
\end{lemma}
\begin{proof}
Since the conforming finite element is used, the proof is the same as Lemma~\ref{LeAnH} for the continuous case.
\end{proof}

Let $u_h$ be the solution of \eqref{HADWeak}. We define the discrete solution operator $S_1^h(\eta): L^2(\partial D) \to V_h$ such that
\[
 u_h = S_1^h(\eta) g
\]
and 
\begin{equation}\label{dnh}
f_1^h = \mathcal{I}S_1^h(\eta) g = u_h|_{\partial D},
\end{equation}
where $\mathcal{I}$ is the restriction of $u_h$ to $V_h^B$.

The following result of the error estimate is standard \cite{BrennerScott2008}. For completeness, we present a proof
for it.
\begin{theorem}\label{f1hf}
Let $\eta \in \mathbb C_0^+$ and assume that the solution of \eqref{HelmholtzAW} $u \in H^2(D)$. 
Let $f_1 = \mathcal{I}S_1(\eta) g, \, g \in L^2(\partial D)$.
Then
\[
\|f_1^h - f_1\|_{L^2(\partial D)} = Ch^{3/2}\|g\|_{L^2(\partial D)}.
\]
\end{theorem}
\begin{proof} Let $u$ and $u_h$ be the solutions for \eqref{HelmholtzA} and \eqref{HADWeak}, respectively.
The Galerkin orthogonality is
\[
a(u-u_h, v_h) = 0 \quad \text{for all } v_h \in V_h.
\]
Using the boundedness of $a(\cdot, \cdot)$ and the G{\aa}rding's inequality \eqref{Garding}, one has that
\begin{eqnarray} \nonumber
\alpha_0 \|u-u_h\|^2_{H^1(D)} & \le & |a(u-u_h, u-u_h) + K(u-u_h, u-u_h)| \\
\nonumber &=&| a(u-u_h, u-v_h) + K\|u-u_h\|^2_{L^2(D)} | \\
\label{a0uuh} &\le& C \|u-u_h\|_{H^1(D)} \|u-v_h\|_{H^1(D)}  + K\|u-u_h\|^2_{L^2(D)}.
\end{eqnarray}

Let $w$ be the solution to the adjoint problem
\begin{equation}\label{AP}
a(v, w) = ( u-u_h, {v} )_{D} \quad \text{for all } v \in V. 
\end{equation}
Then $w \in H^2(D)$ and, for any $w_h \in V_h$, one has that
\begin{eqnarray*}
(u-u_h, u-u_h)  &=& a(u-u_h, w) \\
		&=& a(u-u_h, w-w_h) \\
		&\le& C\|u-u_h\|_{H^1(D)} \|w-w_h\|_{H^1(D)}\\
		&\le& Ch \|u-u_h\|_{H^1(D)} |w|_{H^2(D)} \\
		&\le& Ch \|u-u_h\|_{H^1(D)} \|u-u_h\|_{L^2(D)},
\end{eqnarray*}
where we have used the regularity of the solution for the adjoint problem \eqref{AP}. Consequently, it holds that
\begin{equation}\label{uuhL2}
\|u-u_h\|_{L^2(D)} \le C h \|u-u_h\|_{H^1(D)}.
\end{equation}

Plugging the above inequality in \eqref{a0uuh}, one has that 
\begin{equation}\label{ErrH1}
\alpha_0 \|u-u_h\|^2_{H^1(D)}  \le C \|u-u_h\|_{H^1(D)} \|u-v_h\|_{H^1(D)}  + CKh^2 \|u-u_h\|^2_{H^1(D)}.
\end{equation}
For $h$ small enough, we obtain
\[
\|u-u_h\|_{H^1(D)} \le C \inf_{v \in V_h} \|u-v_h\|_{H^1(D)} \quad \text{for all } v_h \in V_h
\]
and thus
\begin{equation}\label{uuhH1}
\|u-u_h\|_{H^1(D)} \le C h \|u\|_{H^2(D)}.
\end{equation}
Using the trace theorem (Theorem 1.6.6 of \cite{BrennerScott2008}), \eqref{uuhH1} and \eqref{uuhL2}, one obtains
\begin{eqnarray*}
 \|(u-u_h)|_{\partial D}\|_{L^2(\partial D)} 
	&\le& C \|u-u_h\|_{L^2(D)}^{1/2}  \|u-u_h\|_{H^1(D)}^{1/2} \\
	&\le& Ch^{3/2}\|g\|_{L^2(\partial D)}.
\end{eqnarray*}
Hence
\[
\|f_1^h - f_1\|_{L^2(\partial D)} = Ch^{3/2}\|g\|_{L^2(\partial D)}
\]
and the proof is complete.
\end{proof}

Setting $A=I$ and $n(x)=1$, consider the problem of find $u^0_h \in V_h$ such that
\begin{equation}\label{HADWeakB}
(\nabla u^0_h,\nabla v_h)_\Omega-i\langle u^0_h, v_h\rangle_\Gamma-\eta(u^0_h, v_h)_\Omega=\langle p_h g,v_h\rangle_\Gamma \quad \text{for all } v_h \in V_h.
\end{equation}
Similarly, we can define a solution operator $S_0^h(\eta): L^2(\partial D) \to V_h$ for \eqref{HADWeakB} by
\[
u^0_h = S_0^h g 
\]
and
\begin{equation}\label{d1h}
f_0^h = \mathcal{I}S_0^h(\eta) g := u^0_h|_{\partial D}.
\end{equation}
From Theorem~\ref{f1hf}, one has that
\begin{equation}\label{f0hf}
\|f_0^h - f_0\|_{L^2(\partial D)} = Ch^{3/2}\|g\|_{L^2(\partial D)}, \quad  g \in L^2(\partial D),
\end{equation}
where $f_0 = \mathcal{I}S_0(\eta) g, \, g \in L^2(\partial D)$.

Let $T_h(\eta)$ be an finite element approximation for $T(\eta)$ given by
\[
T_h(\eta) : = I_h(S_1^h(\eta) - S_0^h(\eta)),
\]
where $I_h: V_h \to V_h^B$ is the restriction operator. 

\begin{lemma}\label{pnp}
If $g \in H^1(\partial D)$, the projection $p_n \in \mathcal{L}(L^2(\partial D), V_h^B)$ defined in \eqref{ph} satisfies
\[
\lim_{h \to 0} \|p_h g\|_{L^2(\partial D)} = \|g\|_{L^2(\partial D)}.
\]
\end{lemma}
\begin{proof} For simplicity, assume that $V_h^B$ is the linear Lagrange element space $V_h^B$. One has that (see Section 3.2 of \cite{SunZhou2016})
\[
\inf_{v_h \in V_h^B} \|g-v_h\|_{L^2(\partial D)} \le C h \|g\|_{H^1(\partial D)}.
\]
Thus 
\[
\|p_h g - g\|_{L^2(\partial D)} \to 0 \quad \text{as } h \to 0.
\]
\end{proof}

\begin{lemma}\label{suphsupO} Let $h_0 >0$ be small enough.
For every compact set $\Omega \subset \mathbb C_0^+$,
\begin{equation}\label{Fhnorm}
\sup_{h < h_0} \sup_{\eta \in \Omega} \|T_h(\eta)\| < \infty.
\end{equation}
\end{lemma}
\begin{proof} Fix $\eta \in \Omega$ and assume $h < h_0$.
Let $g_h \in V_h^B$. Write $p_h g$ as $g_h$ in \eqref{HADWeak} and let $u_h$ be the solution of \eqref{HADWeak}.
One has that
\[
\|u_h\|_{H^1(D)} \le C\|g_h\|_{L^2(\partial D)},
\]
where $C$ does not depend on $h$ but does depend on $\eta$.
Thus
\[
\|f_1^h \|_{L^2(\partial D)} \le C\|g_h\|_{L^2(\partial D)}.
\]
Similarly, $\|f_0^h\|_{L^2(\partial D)} \le C\|g_h\|_{L^2(\partial D)}$. Then we have that
\[
\|T_h(\eta)g_h\|_{L^2(\partial D)} \le C \|g_h\|_{L^2(\partial D)}.
\]
Since $\Omega$ is compact, \eqref{Fhnorm} holds.
\end{proof}

\begin{lemma}\label{ThphgphTg} Assume that $g \in H^1(\partial D)$ and, for $\eta \in \Omega$, $T(\eta)g \in H^1(\partial D)$. Then
\begin{equation}\label{h20ThphgphTg}
\lim_{h \to 0} \|T_h(\eta) p_h g - p_h T(\eta) g\|_{L^2(\partial D)} = 0.
\end{equation}
\end{lemma}
\begin{proof} Using the triangle inequality, Theorem \ref{f1hf}, and Lemma \ref{pnp},  one has that
\begin{eqnarray*}
&& \|T_h(\eta) p_h g - p_h T(\eta) g\|_{L^2(\partial D)} \\
&=& \|T_h(\eta) p_h g - T(\eta) g - p_h T(\eta) g +T(\eta) g\|_{L^2(\partial D)} \\
&\le& \|T_h(\eta) p_h g - T(\eta) g\|_{L^2(\partial D)} + \|p_h T(\eta) g -T(\eta) g\|_{L^2(\partial D)} \\
&\le& \| f_1^h - f_1\|_{L^2(\partial D)} + \| f_0^h - f_0\|_{L^2(\partial D)} + \|p_h T(\eta) g -T(\eta) g\|_{L^2(\partial D)} \\
&\le& Ch \|g\|_{H^1(\partial D)}.
\end{eqnarray*}
Then \eqref{h20ThphgphTg} follows immediately.
\end{proof}

%

Now we are ready to present the main convergence result. To this end, we make the following assumption.

{\bf Assumption:} There exist two Sobolev spaces $X$ and $Y$, $X \subset L^2(\partial D)$ and $ Y \subset H^{1/2}(\partial D)$, such that
$T: \Omega \to \mathcal{L}(X, Y)$ is a holomorphic Fredholm operator function of index zero.

{\bf Remark:} 
We refer the readers to \cite{CakoniKress2017} (Theorem 3.5 therein) for a similar result using the boundary integral equation
method for the transmission eigenvalue problem of isotropic media. By assuming that $\partial D$ is $C^{2,1}$ and $A=I$ and $n(x) = n_c$, 
the authors show that a boundary integral operator similar to $T(\eta)$ is a Fredholm operator with index zero from 
$X=H^{-3/2}(\partial D)$ to $Y=H^{3/2}(\partial D)$.

\begin{theorem}\label{mainThm}
	Let $\lambda \in \sigma(T)$ and $h$ be small enough. Assume that $\mathcal{N}(T(\lambda)) \subset H^1(\partial D)$. 
	There exist $\lambda_h \in \sigma(T_h)$ such that
	$\lambda_h \to \lambda$ as $h \to 0$.
	 For any sequence $\lambda_h \in \sigma(T_h)$,
	the following estimate holds:
	\begin{equation}\label{lambdaconv}
		|\lambda_h - \lambda| \le C h^{\frac{1}{r_0}},
	\end{equation}
	where $r_0:=\kappa(T, \lambda)$.
\end{theorem}

\begin{proof}
	Let $\{h_n\}$ be a small enough monotonically decreasing sequence of positive numbers and $h_n \to 0$ as $n \to \infty$.
	Then we have a sequence of operators $T_n(\lambda) := T_{h_n}(\lambda)$, finite element spaces $V_n:=V_{h_n}$, 
	$V_n^B:=V_{h_n}^B$, and the projection $p_n := p_{h_n}$. 
	Clearly, we have that 
	\[
		\lim_{n \to \infty} \|p_n g\|_{L^2(\partial D)}=\|g\|_{L^2(\partial D)}.
	\]
	Thus Assumption (A1) in Section 2 is satisfied since $X, Y \subset L^2(\partial D)$, $q_n=p_n$. 
	Assumption (A2) holds due to Lemma~\ref{suphsupO}.
	Assumption (A3)(a) holds due to Lemma~\ref{ThphgphTg}.
	
	Next we verify Assumption (A3)(b). 
	Let $N'\subset \mathbb N$ and $v_{n} \in V_n^B, n \in N'$ be a  subsequence with $\|v_n\|_{L^2(\partial D)}$
	bounded and 
	\begin{equation}\label{TNvNpNy}
	\lim_{n \to \infty} \|T_n(\lambda) v_n - p_n y\|_{L^2(\partial D)} = 0
	\end{equation}
	for some $y \in L^2(\partial D)$.  
	We shall show that there exists
	a subsequence $N'' \subset N'$ and a $v \in L^2(\partial D)$ such that 
	\begin{equation}\label{assumptionA3b}
	\lim_{N'' \ni n \to \infty}\|v_n - p_n v\|_{L^2(\partial D)} = 0.
	\end{equation}
	
	If $\lambda \in \rho(T)$, then $T(\lambda)^{-1}$ exists and is bounded. Let $v =  T(\lambda)^{-1}y$. Due to \eqref{TNvNpNy},
	one has that $T_n(\lambda) v_n \to y$ as $n \to \infty$. For $n$ large enough, $T_n(\lambda)^{-1}$ exists and is bounded.
	Hence $v_n \to T_n(\lambda)^{-1} p_n y$. Together with the fact that $T_n(\lambda)^{-1} p_n y \to v$ (Assumption (A3)(a)), we obtain that
	$v_n \to p_n v$ as $n \to \infty$.

	If $\lambda \in \sigma(T)$, let $G(\lambda)$ denote the associated generalized eigenspace.
	Then consider $T(\lambda): L^2(\partial D) / G(\lambda) \to \mathcal{R}(T) \subset L^2(\partial D)$,
	where $\mathcal{R}(T)$ is the range of $T$. Then $T(\lambda)$ has a bounded inverse 
	from $\mathcal{R}(T)$ to $L^2(\partial D)/G(\lambda)$. 
	Since $y \in \mathcal{R}(T)$, let
	$v' = T(\lambda)^{-1}y \in L^2(\partial D)/G(\lambda)$. Let $\hat{v}_n = (v_n -v')|_{G(\lambda)}$. 
	Since $G(\lambda)$ is finite dimensional, $\hat{v}_n$ has a convergence subsequence,
	denoted by $\hat{v}_{n'}$. Then $v = \lim_{n' \to \infty} \hat{v}_{n'} + v'$ satisfies \eqref{assumptionA3b}.
	
	The quantity $\epsilon_h$ is the consistency error defined by
	\[
		\epsilon_h = \max_{|\lambda-\lambda_0| \le \delta}\max_{g \in G(\lambda)} \| T_h(\lambda)p_h g - p_h T(\lambda) g\|_{L^2(\partial D)}.
	\]
	where  $\delta > 0$ is chosen sufficiently small.
	From the proof of Lemma \ref{ThphgphTg}, one clearly has that
	\[
		\epsilon_h \le C h
	\]
	and \eqref{lambdaconv} follow Theorem~\ref{Thm210} directly.
\end{proof}

\begin{corollary}
For a simple eigenvalue $\lambda$, there exist $\lambda_h \in \sigma(T_h)$ such that
\[
|\lambda_h - \lambda| \le C h.
\]
\end{corollary}

\section{Spectral Indicator Method}
To compute the eigenvalues of $T_h$ in a bounded simply connected region $\Omega \subset \mathbb C$, we propose a new algorithm based on spectral projection.
It is an extension of the spectral indicator method proposed in \cite{HuangEtal2016JCP, HuangEtal2017} 
to compute the generalized eigenvalues of non-Hermitian matrices.

Without loss of generality,  let $\Omega \subset \mathbb C$ be a square and 
$\Gamma$ be the circle circumscribing $\Omega$ (see Fig.~\ref{discs}).
Assume that $T_h(\eta)^{-1}$ exists and thus is bounded for all $\eta \in \Gamma$.
Define an operator $P: V_h^B \to V_h^B$ by
\[
P=\dfrac{1}{2\pi i}\int_{\Gamma}T_h(\eta)^{-1}d\eta.
\]
Let $v_h$ be an arbitrary (random) function in $V_h^B$.
If $T_h(\eta)$ has no eigenvalues in $\Omega$, $Pv_h = 0$ for any $v_h \in V_h^B$.
On the other hand, if $T_h(\eta)$ has eigenvalues in $\Omega$, $Pv_h \ne 0$ almost surely.
This is the basic idea behind the spectral indicator method.
In this section, we develop a variation of the spectral indicator method to compute the eigenvalues of $T_h$ in $\Omega$.

Assume that the Lagrange basis functions for $V_h$ is given by 
\[
\phi_i, i=1, \ldots, N_B, N_B+1, \ldots, N
\]
and $\phi_i|_{\partial D}, i=1, \ldots, N_B,$ are the basis functions for $V_h^B$.
Let $A^1_h$, $M_h^B$, $M_h^n$ be the matrices corresponding to the terms
\[
(A\nabla u_h,\nabla v_h)_D, \quad \langle u_h, v_h\rangle_{\partial D}, \quad (nu_h, v_h)_D
\]
in \eqref{HADWeak}, respectively. Let $A^0_h$ and $M_h$ be the matrices corresponding to
$(\nabla u_h,\nabla v_h)_D$ and $(u_h, v_h)_D$ in \eqref{HADWeakB}, respectively.

For $\eta \in \Gamma$, define 
\[
R^1_h = (A^1_h - i M_h^B - \eta M^n_h)^{-1} M_B
\] 
and 
\[
R^0_h = (A^0_h - i M_h^B - \eta M_h)^{-1} M_B, 
\]
where $M_B$ is the matrix such that $(M_B)_{i,j} = (v_j, v_i), v_j \in V_h^B, v_i \in V_h$. 
Thus $M_B: V_h \to V_h^B$ is an $N \times N_B$ projection matrix.
Denote by $M_B^t$ be the transpose of $M_B$.
Then the matrix version of the operator eigenvalue problem is to find $\eta$ and a nontrivial $g_h \in V_h^B$ such that
\begin{equation}\label{ThetaM}
T_h(\eta) g_h := M_B^t(R^1_h - R^0_h)M_B g_h = 0.
\end{equation}

Let $f_h \in V_h^B$ be a random function and $x_h$ be the solution of 
$T_h(\eta) x_h(\eta) = {f}_h$.
Using the trapezoidal rule to approximate the integral
\[
Pf_h = \frac{1}{2\pi i} \int_\Gamma x_h(\eta) d \eta,
\]
we define an indicator for $\Omega$ as
\[
I_\Omega := \left| \frac{1}{2\pi i} \sum_{j=1}^{n_0} \omega_j x_h(\eta_j) \right|,
\]
where $n_0$ is the number of quadrature points and $\omega_j$'s are the weights. 
The indictor $I_\Omega$ is used to test if $\Omega$ contains eigenvalue(s) or not.
If $I_\Omega > 0$, there are eigenvalues in $\Omega$. Then $\Omega$ is (uniformly) subdivided into smaller squares. 
The indicators of these regions are computed. The procedure continues until the size of the squares is smaller than a specified precision $\epsilon_0$, 
say, $10^{-6}$. Then the centers of the squares are the approximations of the eigenvalues of $T_h$ (see Fig.~\ref{discs}). 

\begin{figure}
\begin{center}
{ \scalebox{0.5} {\includegraphics{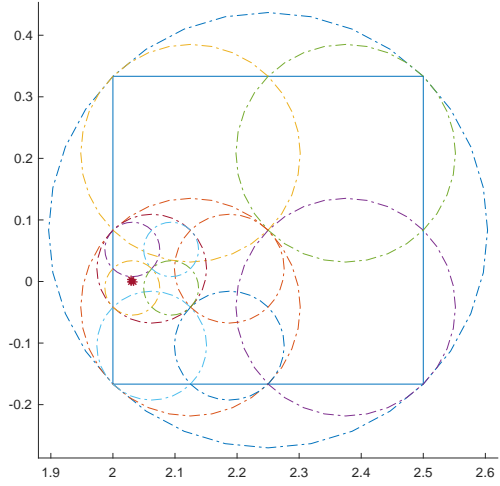}}}
\caption{Several levels of discs on $\mathbb C$ the algorithm SIM-H checked. '*' denotes the eigenvalue.}
 \label{discs}
\end{center}
\end{figure}


The following algorithm SIM-H (spectral indicator method for holomorphic functions)  approximates the eigenvalues of $T$ in $\Omega$

\begin{itemize}
\item[] {\bf SIM-H:}
\item[-] Give a domain $D$.
\item[-] Given a square $\Omega$, the precision $\epsilon_0$, the indicator threshold $\delta_0$.
\item[1.] Generate a triangular mesh for $D$ and the matrices $A^1_h, M_h^B, M_h^n, A^0_h, M_h$.
\item[2.] Choose a random $f \in V_h^B$.
\item[3.] While the length of the square $d > \epsilon_0$, do
	\begin{itemize}
	\item For each square $\Omega_i$ at current level, evaluate the indicator $I_{\Omega_i}$:
		\[
			I_{\Omega_i} := \frac{1}{2\pi i} \sum_{j=1}^{n_0} \omega_j f'_h \left(M_B^t(R^1_h - R^0_h)M_B\right)^{-1} f_h.
		\]
	\item If $| I_{\Omega_i}| < \delta_0$, uniformly divide $\Omega_i$ into smaller squares.
	\end{itemize}
\item[4.] Output the eigenvalues (centers of the small squares).
\end{itemize}
\vskip 0.2cm

The above algorithm computes the eigenvalues up to a given precision $\epsilon_0$.
If the multiplicity of an eigenvalue $\lambda$ is further needed, one can let $\Gamma = \{ z : |z-\lambda|=\epsilon_0\}$,
i.e., the circle centered at $\lambda$ with radius $\epsilon_0$.  
Assume that there are $\kappa$ eigenvalues, counting multiplicity, $\lambda_1, \lambda_2, \ldots, \lambda_\kappa$ inside $\Gamma$.
Choose $m > \kappa $ linearly independent functions ${w}_h^j \in V_h^B, j=1, \ldots, m$.
Let $x_j(\eta) \in V_h^B$ solve
\[
T_h(\eta) x_j(\eta) = {w}_h^j, \quad \eta \in \Gamma, \quad j = 1, \ldots, m.
\]
Let $M_0$ be the $m \times m$-matrix valued function given by
\[
M_0 = V^T V, \quad V=[x_1(\eta), x_2(\eta), \ldots, x_m(\eta)].
\]
Then one has that (see, e.g., \cite{Beyn2014})
\begin{equation}\label{kappa}
\kappa = \text{rank}(M_0).
\end{equation}

Hence one can pick up $m > \kappa$ basis functions in $V_h^B$ and
evaluate $M_0$. Then the multiplicity is the number of significant singular values of $M_0$. Since the eigenvalues are already isolated up to 
the precision $\epsilon$, $m$ can be a small integer, say $3$. If it is not enough, i.e., $\kappa = m$, increase $m$ until
$\text{rank}(M_0) < m$.

It is also possible to compute the eigenfunctions if the eigenvalues are known. We refer the readers
to \cite{Karma1996a, Karma1996b, Beyn2014} for more information.

%

\section{Numerical Examples}
We present some numerical examples to show the effectiveness of the proposed method.
Consider two domains in $\mathbb R^2$, a disc defined by
\[
D_1=\left\{ (x,y)  \in \mathbb R^2 \, | \, x^2+y^2 < \frac{1}{4} \right\}
\]
and the square defined by
\[
D_2=\left\{ (x,y) \in \mathbb R^2 \, | \, |x|+|y| < 1 \right\}.
\]
For all examples, we set the precision $\epsilon_0 = 10^{-6}$ in SIM-H.

\subsection{Example 1}
Let  $A=I$, $n=16$. We generate a triangular mesh with $h\approx 1/40$ and
use the linear Lagrange element for discretization. For $D_1$, the region in which we compute the eigenvalues is 
\[
\Omega:=\{a+ib \,|\, 4.2<a<5.4, -0.6 < b < 0.6\}.
\]
For $D_2$, we set 
\[
\Omega:=\{a+ib \,|\, 2.8<a<3.8, -0.5 < b < 0.5\}.
\] 
In Fig.~\ref{distribution}, we show the computed eigenvalues. These eigenvalues are consistent with the results in literature, e.g., \cite{ColtonMonkSun2010}.
\begin{figure}
\begin{center}
\begin{tabular}{cc}
\resizebox{0.52\textwidth}{!}{\includegraphics{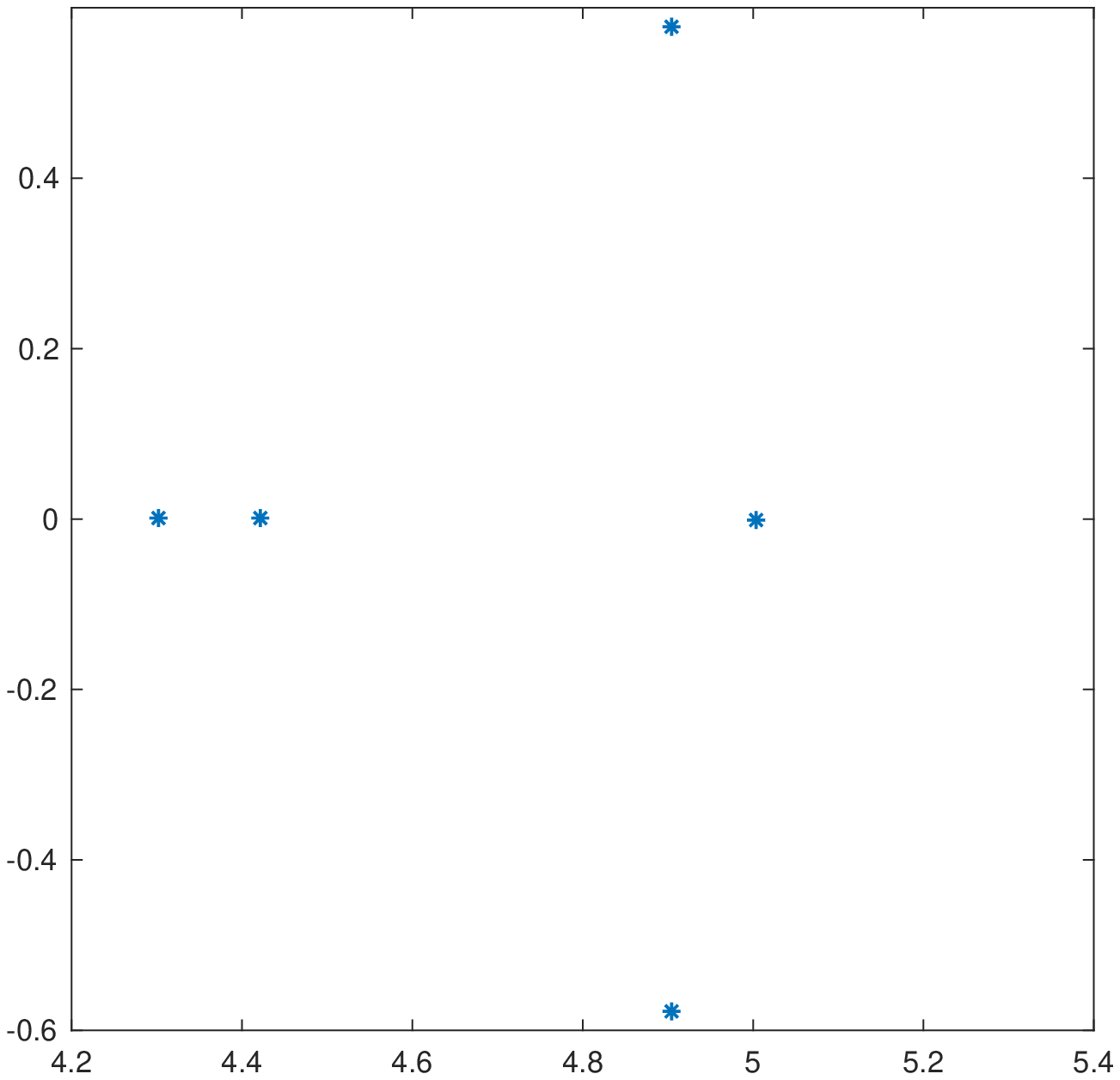}}&
\resizebox{0.52\textwidth}{!}{\includegraphics{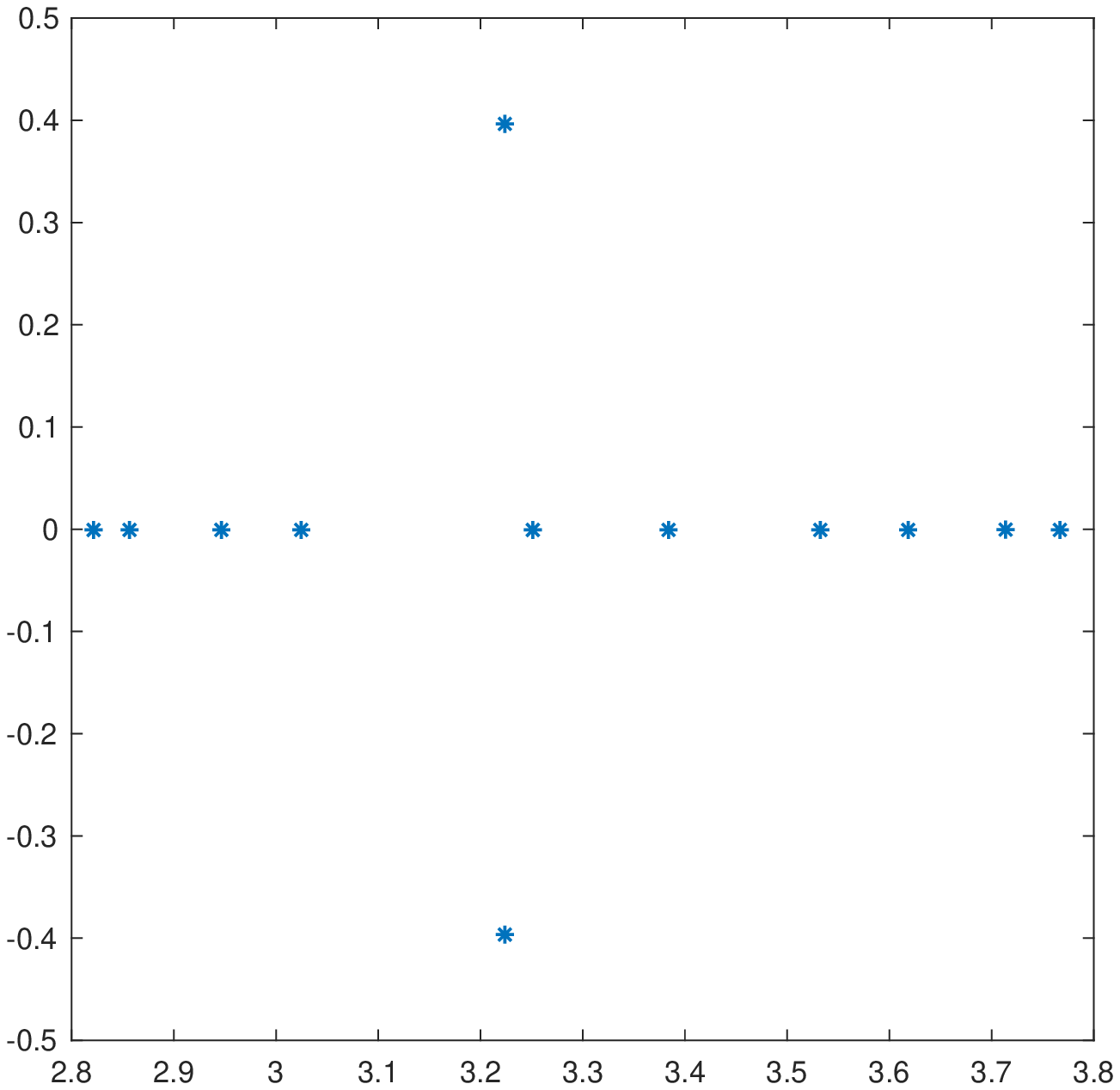}}
\end{tabular}
\end{center}
\caption{Distributions of transmission eigenvalues when $A=I$ and $n=16$. 
Left: $D=D_1$ and $\Omega:=\{a+ib \,|\, 4.2<a<5.4, -0.6 < b < 0.6\}$.  
Right: $D=D_2$ and $\Omega:=\{a+ib \,|\, 2.8<a<3.8, -0.5 < b < 0.5\}$.}
\label{distribution}
\end{figure}

Denote by $\lambda_1^R$ the smallest positive transmission eigenvalue and by $\lambda_1^C$ the complex transmission
eigenvalue with smallest norm and positive real part. In Table~\ref{convergenceC}, we show the computed eigenvalues on
four uniformly refined meshes with the mesh size $h_1 \approx 1/10$ for the coarsest mesh. 
Denote the sequence of the computed eigenvalues $\lambda_{h_i}, i=1,2,3, 4$, which approximate an eigenvalue $\lambda$.
Define the relative error 
\begin{equation}\label{relerr}
	err_i = \frac{|\lambda_{h_i} - \lambda_{h_{i+1}}| }{|\lambda_{h_i}|}, \quad i=1, 2, 3.
\end{equation}
We show the convergence orders in Table~\ref{convergenceC} for $D_1$ and in Table~\ref{convergence} for $D_2$.
Second order convergence is obtained and
the eigenvalues are consistent with the values in \cite{ColtonMonkSun2010, Sun2011SIAMNA}.

\begin{table}[h!]
\caption{Computed eigenvalues for $D_1$, relative errors, and convergence orders.}
\begin{center}
\begin{tabular}{lllllll}
\hline
$h$ & $\lambda_1^R$& $err$ & order &$\lambda_1^R$ &$err$&order\\
\hline
$\frac{1}{10}$ &2.030145&-&-&5.066611 + 0.487817i&-\\
$\frac{1}{20}$&1.998634&0.015521 &-&4.948646 + 0.574916i&0.028808&-\\
$\frac{1}{40}$&1.990651&0.003994&1.96&4.912988 + 0.578294i&0.007189 &3.17\\
$\frac{1}{80}$ &1.988659& 0.001001 &2.00&4.903908 + 0.578200i&0.001836&2.01\\
\hline
\end{tabular}
\end{center}
\label{convergenceC}
\end{table}


\begin{table}[h!]
\caption{Computed eigenvalues for $D_2$, relative errors, and convergence orders.}
\begin{center}
\begin{tabular}{lllllll}
\hline
$h$ & $\lambda_1^R$& $err$ & order &$\lambda_1^C$ &$err$&order\\
\hline
$\frac{1}{10}$&1.367587&-&-&3.020708 + 0.214643i&-\\
$\frac{1}{20}$&1.338741&0.021092 &-&3.253545 + 0.380979i&0.094491&-\\
$\frac{1}{40}$&1.331498&0.005410&1.96&3.223205 + 0.397238i&0.010508&2.00\\
$\frac{1}{80}$&1.329679&0.001366&1.99&3.214951 + 0.399112i&0.002606&1.97\\
\hline
\end{tabular}
\end{center}
\label{convergence}
\end{table}

%

%

\subsection{Example 2}
We compute transmission eigenvalues for anisotropic media. Let $n=1$ and set 
\[
(\text{I})\, A_1 = \left(\begin{array}{cc} 1/2&0\\ 0&1/8 \end{array}\right), 
\quad (\text{II})\, A_2=\left(\begin{array}{cc} \frac{x^2+y^2}{2}&0\\ 0&\frac{2-x^2-y^2}{8} \end{array}\right).
\]
The mesh size is $h \approx 1/40$.
We list the computed eigenvalues in Table~\ref{ComEigA}. Note that $\lambda_1^R(D_1)$ is consistent with
the value in Section 5.1 of \cite{JiSun2013JCP}.

\begin{table}[h!]
\caption{Computed eigenvalues with mesh size $h\approx 1/40$.}
\begin{center}
\begin{tabular}{lllll}
\hline
& $\lambda_1^R(D_1)$& $\lambda_1^C(D_1)$ & $\lambda_1^R(D_2)$ &$\lambda_1^C(D_2)$ \\
\hline
$A_1$&4.880159&3.999540 + 1.569577i&3.334459&2.420045 + 1.362148i\\
$A_2$&2.359596&3.689864 + 1.416741i &1.976699&2.397842 + 0.922280i\\
\hline
\end{tabular}
\end{center}
\label{ComEigA}
\end{table}

%
%
%
%
%
%

\section{Conclusions and Future Work}
In this paper, we develop a finite element method for the nonlinear transmission eigenvalue problem.
Using the approximation theory for eigenvalues of holomorphic operator functions, we prove the convergence. 
A new spectral indicator method is designed to compute the eigenvalues. Numerical
examples show the effectiveness of the proposed method.

The convergence order proved in Theorem~\ref{mainThm} seems to be suboptimal. 
The theory provides a lower bound for the convergence order.
The algorithm SIM-H needs to solve many source problems and is computationally expensive.
In future, we plan to develop a parallel version of SIM-H.

The framework using the approximation theory for eigenvalues of holomorphic operator functions 
can be used to prove the convergence of finite element methods for a large class of nonlinear eigenvalue
problems of partial differential equations. For example, 
the method can be extended to compute the nonlinear transmission eigenvalue problem of the Maxwell's equations
for anisotropic media. Another example is the scattering resonances for frequency dependent material properties.
These problems are currently under our investigation.


\end{document}